\documentclass{amsart}

\usepackage{amsmath,amssymb,amsfonts, amscd}

\newcommand{\Hom}{\normalfont\mbox{Hom}\,}
\newcommand{\Ext}{\normalfont\mbox{Ext}\,}
\newcommand{\Tor}{\normalfont\mbox{Tor}\,}
\newcommand{\pd}{\normalfont\mbox{pd}\,}
\newcommand{\id}{\normalfont\mbox{id}\,}
\newcommand{\GP}{\mathcal{GP}\,}
\newcommand{\GI}{\mathcal{GI}\,}
\newcommand{\fd}{\normalfont\mbox{fd}\,}
\newcommand{\Ker}{\normalfont\mbox{Ker}\,}
\newcommand{\gd}{\mbox{Gdim}\,}
\newcommand{\coker}{\normalfont\mbox{coker}\,}

\newcommand{\wdim}{\normalfont\mbox{wdim}\,}

\newcommand{\gldim}{\normalfont\mbox{gldim}\,}
\newcommand{\wGgldim}{\normalfont\mbox{wGgldim}\,}
\newcommand{\Ggldim}{\normalfont\mbox{Ggldim}\,}
\newcommand{\Gpd}{\normalfont\mbox{Gpd}\,}
\newcommand{\Gid}{\normalfont\mbox{Gid}\,}

\newcommand{\FPD}{\normalfont\mbox{FPD}\,}
\newcommand{\Gfd}{\normalfont\mbox{Gfd}\,}
\newcommand{\im}{\normalfont\mbox{Im}\,}

\theoremstyle{plain}
\newtheorem{theorem}{Theorem}[section]

\newtheorem{proposition}[theorem]{Proposition}
\newtheorem{corollary}[theorem]{Corollary}

\theoremstyle{definition}
\newtheorem{definition}[theorem]{Definition}

\theoremstyle{Definition and Notation}

\catcode`\ç=13
\defç{\c{c}}
\catcode`\é=13
\defé{\'e}
\catcode`\à=13
\defà{\`a}
\catcode`\è=13
\defè{\`e}
\catcode`\â=13
\defâ{\^a}
\catcode`\ù=13
\defù{\`u}
\catcode`\ê=13
\defê{\^e}
\catcode`\î=13
\defî{\^\i}
\catcode`\ô=13
\defô{\^o}
\begin{document}

\title[]{The Right Orthogonal Class $\GP(R)^{\perp}$ via $\Ext$}

\author{Mohammed Tamekkante}
\address{Mohammed Tamekkante\\Department of Mathematics, Faculty of Science and
Technology
of Fez, Box 2202, University S.M. Ben Abdellah Fez, Morocco.}
\email{tamekkante@yahoo.fr}

\keywords{Cotorsion theory; Gorenstein homological dimensions of modules and rings.}

\subjclass[2000]{13D05, 13D02}
\begin{abstract} In this paper, we study the pair $(\GP(R),\GP(R)^{\perp})$ where $\GP(R)$
is the class of all Gorenstein projective modules. We prove that it is complete hereditary
cotorsion theory provided $l.\Ggldim(R)<\infty$. We discuss also, when every Gorenstein
projective module is Gorenstein flat.
\end{abstract}
\maketitle
\section{Introduction}
Throughout this paper, $R$ denotes a non-trivial associative ring  and all
modules -if not specified otherwise-  are left and unitary. The definitions and
notations employed in this paper are based on those introduced by
Holm in \cite{Holm}.\\
Let $R$ be a ring, and let $M$ be an $R$-module. As usual we use
$\pd_R(M)$, $\id_R(M)$ and $\fd_R(M)$ to denote, respectively, the
classical projective dimension, injective dimension and flat
dimension of $M$. We use also $\gldim(R)$ and $\wdim(R)$ to denote, respectively, the
classical global and weak dimension of $R$.\\

For a two-sided Noetherian ring $R$, Auslander and Bridger
\cite{Aus bri} introduced the $G$-dimension, $\gd_R (M)$, for
every finitely generated $R$-module $M$. They showed that there is
an inequality $\gd_R (M)\leq \pd_R (M)$ for all finite $R$-modules
$M$, and equality holds if $\pd_R (M)$ is finite.

Several decades later, Enochs and Jenda \cite{Enochs,Enochs2}
defined the notion of Gorenstein projective dimension
($G$-projective dimension for short), as an extension of
$G$-dimension to modules which are not necessarily finitely
generated, and the Gorenstein injective dimension ($G$-injective
dimension for short) as a dual notion of Gorenstein projective
dimension. Then, to complete the analogy with the classical
homological dimension, Enochs, Jenda and Torrecillas \cite{Eno
Jenda Torrecillas} introduced the Gorenstein flat dimension. Some
references are
 \cite{Christensen
and Frankild, Enochs, Enochs2, Eno Jenda Torrecillas, Holm}.\\

Recall that an $R$-module $M$ is called
Gorenstein projective if, there exists an exact sequence of
projective  $R$-modules:
$$\mathbf{P}:...\rightarrow P_1\rightarrow P_0\rightarrow
P^0\rightarrow P^1\rightarrow ...$$ such that $M\cong
\im(P_0\rightarrow P^0)$ and such that the operator $\Hom_R(-,Q)$
leaves $\mathbf{P}$ exact whenever $Q$ is a
projective. The resolution $\mathbf{P}$ is called a complete projective resolution. \\
The Gorenstein injective $R$-modules are defined
dually.\\
 And an $R$-module $M$ is called
Gorenstein flat if, there exists an exact sequence of flat
 $R$-modules:
$$\mathbf{F}:...\rightarrow F_1\rightarrow F_0\rightarrow
F^0\rightarrow F^1\rightarrow ...$$ such that $M\cong
\im(P_0\rightarrow P^0)$ and such that the operator $I\otimes_R-$
leaves $F$ exact whenever $I$ is a right
 injective $R$-module. The resolution $\mathbf{F}$ is called complete flat resolution.\\

The  Gorenstein projective, injective and flat
dimensions are defined in term of resolution and denoted by $\Gpd(-)$, $\Gid(-)$ and
$\Gfd(-)$ respectively (see \cite{Christensen, Enocks and janda, Holm}).\\

$\mathbf{Notation.}$ By  $\mathcal{P}(R)$ and  $\mathcal{I}(R)$ we denote the classes of
all projective and
injective  $R$-modules respectively and by $\overline{\mathcal{P}(R)}$ and
$\overline{\mathcal{I}(R)}$ we denote the classes of all modules with finite projective
dimension and injective dimension respectively. Furthermore, we let
 $\mathcal{GP}(R)$ and
$\mathcal{GI}(R)$  denote the classes of all
 Gorenstein projective and  injective  $R$-modules
respectively. The character module $\Hom_{\mathbb{Z}}(M,\mathbb{Q}/\mathbb{Z})$ is
denoted by $M^+$.\\

In \cite{Bennis and Mahdou2}, the authors prove the equality:
$$\sup\{\Gpd_R(M)\mid \text{$M$ is an $R$-module}\}=\sup\{\Gid_R(M)\mid \text{$M$ is an
$R$-module}\}$$
They called the common value of the above quantities the left
Gorenstein global dimension of $R$ and denoted it by
$l.\Ggldim(R)$. Similarly, they set
$$l.\wGgldim(R)=\sup\{\Gfd_R(M)\mid \text{$M$ is an  $R$-module}\}$$
which they called the left weak Gorenstein global dimension of
$R$.\\

Given a class $\mathfrak{X}$ of $R$-module we set:
$$\begin{array}{ccccc}
  \mathfrak{X}^{\perp}&=& \Ker \Ext_R^1(\mathfrak{X},-)&=&\{M\mid \Ext_R^1(X,M)=0 \text{
  for all } X\in \mathfrak{X}\}, \\
  ^{\perp}\mathfrak{X}&=& \Ker \Ext_R^1(-,\mathfrak{X})&=&\{M\mid \Ext_R^1(M,X)=0 \text{
  for all } X\in \mathfrak{X}\}$$
\end{array}$$

\begin{definition}[\textbf{Precovers and Preenvelopes}]
Let $\mathfrak{X}$ be any class of $R$-modules, and let $M$ be an $R$-module.
\begin{itemize}
  \item An
$\mathfrak{X}$-precover of $M$ is an $R$-homomorphism $\varphi:X \rightarrow M$, where
$X\in \mathfrak{X}$, and such that the sequence,
$$\begin{CD}
\Hom_R(X',X)@>\Hom_R(X',\varphi)>> \Hom_R(X',M)@>>>0
\end{CD}$$
is exact for every $X'\in \mathfrak{X}$. An $\mathfrak{X}$-precover is called special
if $\varphi$ is surjective and $\ker(\varphi)\in \mathfrak{X}^{\perp}$.
  \item An
$\mathfrak{X}$-preenvelope  of $M$ is an $R$-homomorphism $\varphi:M \rightarrow X$,
where $X\in \mathfrak{X}$, and such that the sequence,
$$\begin{CD}
\Hom_R(X,X')@>\Hom_R(\varphi,X')>> \Hom_R(M,X')@>>>0
\end{CD}$$
is exact for every $X'\in \mathfrak{X}$. An $\mathfrak{X}$-preenvelope  is called
special if $\varphi$ is injective   and $\coker(\varphi)\in$
$^{\perp}\mathfrak{X}$.
\end{itemize}
\end{definition}
For more details about precovers (and preenvelopes), the reader may consult \cite[Chapters 5 and 6]{Enocks and janda}.
\begin{definition}[\textbf{Resolving classes 1.1} ,\cite{Holm}] For any class
$\mathfrak{X}$ of $R$-modules.
\begin{itemize}
  \item We call $\mathfrak{X}$ projectively resolving if $\mathcal{P}(R)\subseteq
      \mathfrak{X}$, and for every short exact sequence
$0 \longrightarrow X' \longrightarrow X \longrightarrow X"
\longrightarrow 0$ with $X" \in \mathfrak{X}$ the conditions $X'
\in \mathfrak{X}$ and $X \in \mathfrak{X}$ are equivalent.
  \item We call $\mathfrak{X}$ injectively resolving if $\mathcal{I}(R)\subseteq
      \mathfrak{X}$, and for every short exact sequence
$0 \longrightarrow X' \longrightarrow X \longrightarrow X"
\longrightarrow 0$ with $X' \in \mathfrak{X}$ the conditions $X"
\in \mathfrak{X}$ and $X \in \mathfrak{X}$ are equivalent.
\end{itemize}
\end{definition}
A pair $(\mathfrak{X},\mathfrak{Y})$ of classes of $R$-modules is called a cotorsion
theory (\cite{Enocks and janda}) if $\mathfrak{X}^{\perp}=\mathfrak{Y}$ and
$^{\perp}\mathfrak{Y}=\mathfrak{X}$. In this case we call $\mathfrak{X}\cap \mathfrak{Y}$
the kernel of  $(\mathfrak{X},\mathfrak{Y})$. Note that each element $K$ of the kernel is
a splitter in the sense of \cite{Gobel}, i.e., $\Ext_R^1(K,K) = 0$. If $\mathfrak{C}$ is
any class of modules, then $(^{\perp}\mathfrak{C},(^{\perp}\mathfrak{C})^{\perp})$ is easy
seen be a cotorsion theory, called a cotorsion theory generated by $\mathfrak{C}$ (see
please \cite[Definition 1.10]{Trlifaj}). A cotorsion theory $(\mathfrak{X},\mathfrak{Y})$
is called complete (\cite{Trlifaj})  if every $R$-module has a special
$\mathfrak{Y}$-preenvelope and a special $\mathfrak{X}$-precover. A cotorsion
theory $(\mathfrak{X},\mathfrak{Y})$ is said to be hereditary (\cite{Garcia}) if whenever
$0 \rightarrow L'\rightarrow L \rightarrow L'' \rightarrow 0$  is exact with $L,L''\in
\mathfrak{X}$ then $L'$ is also in $\mathfrak{X}$,
or equivalently, if $0\rightarrow M'\rightarrow M  \rightarrow M" \rightarrow 0$ is exact $M',M\in
\mathfrak{Y}$ then $M''$ is also in $\mathfrak{Y}$.\\

\textbf{Note:} Above we have only proved   the results concerning Gorenstein projective
modules. The proof of the Gorenstein injective ones is dual and we can find a dual of
results using in the proofs in \cite{Holm}.

\section{main results}
The aim of this section is the study of the pair $(\GP(R),\GP(R)^{\perp})$.\\
The class $\GP(R)$ verified the following properties.

\begin{theorem}\label{thm1}
For any ring $R$ the following holds:
\begin{enumerate}
 \item $\Ext_R^i(G,M)=0$ for all $i>0$, all $G\in \GP(R)$ and all $M\in
     \GP(R)^{\perp}$.
 \item $\Ext_R^i(M,G)=0$ for all $i>0$, all $G\in \GI(R)$ and all $M\in$
     $^{\perp}\GI(R)$.
\item $\GP(R)^{\perp}$ and $^{\perp}\GI(R)$ are  projectively resolving.
  \item $\GP(R)^{\perp}$ and $^{\perp}\GI(R)$ are injectively resolving.
  \end{enumerate}
\end{theorem}
\begin{proof}
$(1)$. Consider $M\in \GP(R)^{\perp}$. For any Gorenstein projective module $G$ and any
$n>1$ pick an exact sequence $0\rightarrow G' \rightarrow P_1\rightarrow ...\rightarrow
P_n\rightarrow G\rightarrow 0$ where all $P_i$ are projective. Clearly,  $G'$ is also
Gorenstein projective (\cite[Theorem 2.5]{Holm}). So, we have, $\Ext_R^n(G,M)=\Ext_R^1(G',M)=0$, as desired.\\
$(2)$. Dual to $(1)$.\\
$(3)$. We claim that $\GP(R)^{\perp}$ is  projectively resolving. Using the long exact sequence in homology, we conclude that $ \GP(R)^{\perp}$ is
closed by extension, i.e., if $0\rightarrow M\rightarrow M' \rightarrow M'' \rightarrow 0$
where $M$ and $M''$ are in $\GP(R)^{\perp}$ then so is $M'$. Clearly
$\mathcal{P}(R)\subseteq \GP(R)^{\perp}$ (\cite[Proposition 2.3]{Holm}).
Now, consider a short exact sequence $0\rightarrow M\rightarrow M' \rightarrow M''
\rightarrow 0$ where $M'$ and $M''$ are in $\GP(R)^{\perp}$. For an arbitrary Gorenstein projective $R$-module $G$ consider a
short exact sequence $0\rightarrow G \rightarrow P \rightarrow G'\rightarrow 0$ where $P$
is projective and $G'$ is Gorenstein projective (such sequence exists by definition of Gorenstein projective modules). From the long exact sequence
of homology, we have
$$..\rightarrow \Ext_R^1(G',M'')\rightarrow \Ext_R^2(G',M)\rightarrow
\Ext_R^2(G',M')\rightarrow ...$$
Then, $\Ext_R^2(G',M)=0$  since $\Ext_R^1(G',M'')=\Ext_R^2(G',M')=0$ (from $(1)$). Thus, $\Ext_R^1(G,M)=\Ext_R^2(G',M)=0$,
as desired.\\
$(4)$ We claim that $\GP(R)^{\perp}$ is injectively resolving. Clearly, $\mathcal{I}(R)\subseteq  \GP(R)^{\perp}$ and $ \GP(R)^{\perp}$ is closed
by extension. Now, consider a short exact sequence $0\rightarrow M \rightarrow
M'\rightarrow M''\rightarrow 0$ where $M$ and $M'$ are in $ \GP(R)^{\perp}$. Using the
long exact sequence of homology and for all Gorenstein projective module $G$, we have
$$..\rightarrow \Ext_R^1(G,M')\rightarrow \Ext_R^1(G,M'')\rightarrow
\Ext_R^2(G,M)\rightarrow ...$$ Thus, from $(1)$, $\Ext_R^1(G,M'')=0$ for all Gorenstein
projective module as desired.
\end{proof}
Hence, we conclude the following two Corollarys. The second once was proved by Holm in
\cite{Holm}.
\begin{corollary}\label{coro1}For any ring $R$,
\begin{enumerate}
  \item  $\mathcal{P}(R)=\GP(R)\cap \GP(R)^{\perp}$.
  \item $\mathcal{I}(R)=\GI(R)\cap \GI(R)^{\perp}$.
\end{enumerate}

\end{corollary}
\begin{proof} $(1).$ Consider $M\in \GP(R)\cap \GP(R)^{\perp}$ and for such module $M$
consider a short exact sequence $0\rightarrow M' \rightarrow P \rightarrow M \rightarrow
0$ where $P$ is projective. Since $\GP(R)^{\perp}$ is projectively resolving (from Theorem
\ref{thm1}), $M'\in \GP(R)^{\perp}$. Then, $\Ext_R(M,M')=0$ and so this short exact
sequence splits. Therefore, $M$ is a direct summand of $P$ and so projective, as
desired.\\
$(2).$ Dual proof.
\end{proof}
\begin{corollary}
\begin{enumerate}
  \item \cite[Proposition 2.27]{Holm} Every Gorenstein projective (resp., injective)
      module with finite projective (resp., injective) dimension is projective (resp.,
      injective).
  \item Every Gorenstein projective (resp., injective)  module with finite injective
      (resp., projective) dimension is projective (resp., injective).

\end{enumerate}

\end{corollary}
\begin{proof}
$(1).$ If $M$ is a Gorenstein projective module with finite projective dimension, then
$M\in \GP(R)\cap \GP(R)^{\perp}$ (from \cite[Proposition 2.3]{Holm}) and then projective (by Corollary \ref{coro1}). The
injective case is dual.\\
$(2).$ Note that every module $I$ with finite injective dimension  is an element of
$\GP(R)^{\perp}$. Indeed, by definition, for every Gorenstein projective module $G$ we can
find an exact sequence $0\rightarrow G \rightarrow P_{n-1}\rightarrow ...\rightarrow
P_0\rightarrow G' \rightarrow 0$ where all $P_i$ are projective and $G'$ is Gorenstein
projective with $n=\id_R(I)$. Thus, we have $\Ext_R^1(G,I)=\Ext_R^{n+1}(G',I)=0$, as
desired.\\
Now, if $M$ is  Gorenstein projective with finite injective dimension then $M\in
\GP(R)\cap \GP(R)^{\perp}$ and then projective (by Corollary \ref{coro1}) .\\
Dually, we can prove that every module with finite projective dimension is an element of
$^{\perp}\GI(R)$. And then every Gorenstein injective module with finite projective
dimension is injective (by Corollary \ref{coro1}).
\end{proof}
The main result of this paper is the following Theorem:
\begin{theorem}
If $l.\Ggldim(R)<\infty$, then $(\GP(R),\GP(R)^{\perp})$ and $(\GI(R),$$^{\perp}\GI(R))$
are complete hereditary  cotorsion theory.
\end{theorem}
\begin{proof} $(1).$ To prove that $(\GP(R),\GP(R)^{\perp})$  is cotorsion theory, we have
to prove that $^{\perp}(\GP(R)^{\perp})=\GP(R)$. Let $M$ be in $^{\perp}(\GP(R)^{\perp})$.
Thus, $\Ext_R^1(M,N)=0$ for all $N\in \GP(R)^{\perp}$. Since $\Gpd_R(M)<\infty$ and from
\cite[Theorem 2.10]{Holm}, $M$ admits a surjective $\GP(R)$-precover $\varphi:
G\rightarrow M$, where $K=\ker(\varphi)$ satisfies $\pd_R(K)<\infty$. Since $K\in
\GP(R)^{\perp}$, $G$ is a special $\GP(R)$-precover and this short exact sequence splits
since $\Ext_R^1(M,K)=0$. Thus,  $M$ is a direct summand of $G$. Hence, $M$ is Gorenstein
projective (from \cite[Theorem 2.5]{Holm}). Consequently,
$^{\perp}(\GP(R)^{\perp})\subseteq\GP(R)$. The other inclusion is clear. Therefore
$(\GP(R),\GP(R)^{\perp})$ is  cotorsion theory and every $R$-module has a special
$\GP(R)$-precover (and so every $R$-module has a special $\GP(R)^{\perp}$-preenvelope from
\cite[Lemma 1.13]{Trlifaj}).  This imply that $(\GP(R),\GP(R)^{\perp})$ is complete, as
desired.\\
Since $\GP(R)$ is projectively resolving and $\GP(R)^{\perp}$ is injectively resolving,
this torsion is hereditary.\\
$(2).$ To prove the dual Gorenstein injective result, we use \cite[Theorems 2.6 and
2.15]{Holm}.
\end{proof}
\begin{proposition}\label{prop2}
If  $l.\Ggldim(R)<\infty$, then, 
$$\GP(R)^{\perp}=\overline{\mathcal{P}(R)}=\overline{\mathcal{I}(R)}=\;^{\perp}\GI(R).$$
\end{proposition}
\begin{proof} Clearly $\overline{\mathcal{P}(R)}\subseteq \GP(R)^{\perp}$ (from \cite[Proposition 2.3]{Holm}). Now, let $M\in
\GP(R)^{\perp}$ and $N$ an arbitrary $R$-module and set $n=l.\Ggldim(R)$. Then,
$\Gpd_R(N)\leq n$. So, from \cite[Theorem 2.20]{Holm}, we can find an exact sequence
$$0\rightarrow G \rightarrow P_n\rightarrow ... \rightarrow P_1\rightarrow N  \rightarrow
0$$ where all $P_i$ are projective and $G$ is Gorenstein projective. Thus,
$\Ext_R^{j+n}(N,M)=\Ext_R^j(G,M)=0$ for all $j>0$ (by Theorem \ref{thm1}). Then,
$\id_R(M)\leq n$. Using \cite[Corollary 2.7]{Bennis and Mahdou2},
$\overline{\mathcal{P}(R)}=\overline{\mathcal{I}(R)}$ since $l.\Ggldim(R)<\infty$. Then, $M\in
\overline{\mathcal{P}(R)}$. Similarly, we have $^{\perp}\GI(R)=\overline{\mathcal{I}(R)}$.
This complete the proof.
\end{proof}
Recall that the finitistic projective dimension of $R$ is the global dimension defined
as:
$$\FPD(R)=\sup\{\pd_R(M)\mid \text{M is an $R$-module with } \pd_R(M)<\infty\}$$
\begin{proposition}\label{prop3}
If $\GP(R)=$$^{^{\perp}}(\overline{\mathcal{P}(R)})$ and
$\GP(R)^{\perp}=\overline{\mathcal{P}(R)}$, then $\FPD(R)=l.\Ggldim(R)$.
\end{proposition}
\begin{proof}
From \cite[Theorem 2.2]{Trlifaj}, every $R$-module admits a special
$\GP(R)^{\perp}$-preenvelope. On the other hand, by hypothesis, $(\GP(R),\GP(R)^{\perp})$
is  the cotorsion theory generated by $\overline{\mathcal{P}(R)}$ (see \cite[Defintion
1.10]{Trlifaj}). Then, $(\GP(R),\overline{\mathcal{P}(R)})$ is a cotorsion theory. Thus,
from \cite[Lemma 1.13]{Trlifaj}, every $R$-module $M$ has a special $\GP(R)$-precover.\\
The inequality $\FPD(R)\leq l.\Ggldim(R)$ follows from \cite[Theorem 2.28]{Holm}. Now,
suppose that $\FPD(R)\leq n$ and let $M$ be an arbitrary $R$-module. We claim prove that
$l.\Ggldim(R)<\infty$. From the first part of the proof, $M$ admits a special
$\GP(R)$-precover. Then, there is an exact sequence $0\rightarrow K \rightarrow G
\rightarrow M \rightarrow 0$ where $G$ is Gorenstein projective and $K\in
\GP(R)^{\perp}=\overline{\mathcal{P}(R)}$. Thus, $\pd_R(K)\leq n$ and so $\Gpd_R(M)\leq
n+1$. Hence, $l.\Ggldim(R)\leq n+1<\infty$. Then, $l.\Ggldim(R)= \sup\{\Gpd_R(M \mid
\Gpd_R(M)<\infty\}=\FPD(R)$ (from \cite[Theorem 2.28]{Holm}), as desired.
\end{proof}
Then, we conclude the following characterization of the left Gorenstein global dimension
provided $\FPD(R)<\infty$.
\begin{corollary} If $\FPD(R)<\infty$, then the following are equivalents:
\begin{enumerate}
  \item $l.\Ggldim(R)<\infty$.
  \item $\GP(R)=$$^{^{\perp}}\overline{\mathcal{P}(R)}$ and
      $\GP(R)^{\perp}=\overline{\mathcal{P}(R)}$.
\end{enumerate}
\end{corollary}
\begin{proof}
$(1\Rightarrow 2).$ The first equality follows from \cite[Theorem 2.20]{Holm} and the
second from from Proposition \ref{prop2}.\\
$(2\Rightarrow 1).$ Follows from Proposition \ref{prop3}.
\end{proof}
Now, we discuss the rings over which "\emph{every Gorenstein projective module is
Gorenstein flat}".

\begin{proposition}\label{prop}
For any ring $R$, the following are equivalents:
\begin{enumerate}
  \item Every Gorenstein projective module is Gorenstein flat.
  \item $I^+\in \GP(R)^{\perp}$ for every right injective module $I$.
  \item $F^{++}\in \GP(R)^{\perp}$ for every flat module $F$.
\end{enumerate}
\end{proposition}
\begin{proof}
$(1\Rightarrow2).$ Let $I$ be a right injective $R$-module. Since every Gorenstein
projective $R$-module is Gorenstein flat and by definition of Gorenstein flat module, we
have $\Tor_R^1(I,G)=0$ for all $G\in \GP(R)$. By adjointness, we have
$\Ext_R^1(G,I^+)=(\Tor_R^1(I,G))^+=0$. Hence, $I^+\in \GP(R)^{\perp}$, as desired.\\
$(2\Rightarrow 1)$.Consider a complete projective resolution $$\textbf{P=} ...\rightarrow
P_{-1}\xrightarrow{f_{-1}} P_0\xrightarrow{f_0} P_1 \xrightarrow{f_1} ...$$
 We decompose it into short exact sequences $0\rightarrow G_i\rightarrow P_i \rightarrow
 G'_i \rightarrow 0$ where $G_i=\ker(f_i)$ and $G'_i=\im(f_i)$. From \cite[Observation
 2.2]{Holm}, $G_i$ and $G'_i$ are Gorenstein projective.  Now let $I$ be an injective
 right $R$-module. We have $\Tor_R^1(I,G')=0$ since
 $(\Tor_R^1(I,G'))^+=\Ext_R^1(G',I^+)=0$. Thus, $$0\rightarrow I\otimes_RG'\rightarrow
 I\otimes_RP_i \rightarrow I\otimes_RG'\rightarrow 0$$ is exact. So, $I\otimes_R-$ keep
 the exactness of $\textbf{P}$. Then, it is a complete flat resolution. Consequently,
 every Gorenstein projective module is Gorenstein flat, as desired.\\
$(2\Rightarrow 3).$ Let $F$ be a flat $R$-module. Then, $F^+$ is right
injective. So, $F^{++}\in \GP(R)^{\perp}$.\\
$(3\Rightarrow 2).$ Let $I$ be a right injective $R$-module. There exist a flat $R$-module
$F$ such that $F\rightarrow I^+\rightarrow 0$ is exact. Then, $0\rightarrow
I^{++}\rightarrow F^+$ is exact. But $0\rightarrow I \rightarrow I^{++}$ is exact (by
\cite[Proposition 3.52]{Faith}). Then, $0\rightarrow I \rightarrow F^{+}$ is exact and
then $I$ is a direct summand of $F^+$. Hence, $I^+$ is a direct summand of $F^{++}$. On
the other hand, it is easy to see that $\GP(R)^{\perp}$ is closed under direct summands.
Thus, $I^+\in \GP(R)^{\perp}$, as desired.
\end{proof}
\begin{proposition} For any ring $R$,  $\sup\{\Gfd_R(M)\mid M \text{ is Gorenstein
projective}\}=0\text{ or } \infty$.
\end{proposition}
\begin{proof} Recall that if $\Gfd_R(M)\leq n$, we have $\Tor_R^i(I,M)=0$ for all $i>n$.
Indeed, the case $n=0$ is from the definition of Gorenstein flat modules and the case
$n>0$ is deduced from the first case by the $n$-step projective resolution of $M$.\\
Suppose that $\sup\{\Gfd_R(M)\mid M \text{ is Gorenstein projective}\}=n<\infty$. Then,
$\Ext_R^{n+1}(G,I^+)=(\Tor_R^{n+1}(I,G))^+=0$ for all injective right module $I$ and  all Gorenstein
projective module $G$. But for every Gorenstein projective module $G$ we
can find an exact sequence $0\rightarrow G \rightarrow P_{n-1}\rightarrow ...\rightarrow
P_0 \rightarrow G'  \rightarrow 0$ where all $P_i$ are projective and $G'$ is Gorenstein
projective. Thus, $\Ext_R^1(G,I^+)=\Ext_R^{n+1}(G',I^+)=0$. So, $I^+\in \GP(R)^{\perp}$
for every injective right module $I$. Then, from Proposition  \ref{prop}, every Gorenstein
projective is Gorenstein flat. Consequently, $\sup\{\Gfd_R(M)\mid M \text{ is Gorenstein
projective}\}=0$, as desired.
\end{proof}
A direct consequence of the above Proposition is the following Corollary:
\begin{corollary}
If $l.\wGgldim(R)<\infty$, then every Gorenstein projective $R$-module is Gorenstein
flat.
\end{corollary}

\bibliographystyle{amsplain}

\end{document}